\documentclass[12pt,bezier]{article}
\usepackage{amssymb}
\usepackage{mathrsfs}
\usepackage{amsmath}
\usepackage{amsfonts,amsthm,amssymb}
\usepackage{amsfonts}
\usepackage{graphicx,graphics}
\usepackage{subfigure}
\usepackage[]{caption2}
\usepackage{color}
\usepackage{cite}
\usepackage{hyperref}
\usepackage{indentfirst}
\newtheorem{lem}{Lemma}
\newtheorem{thm}{Theorem}

\begin{document}
\title{A forbidden pair for quasi 5-contractible edges}

\author{Shuai Kou$^a$, \quad Weihua Yang$^a$\footnote{Corresponding author. E-mail: ywh222@163.com,~yangweihua@tyut.edu.cn}, \quad Mingzu Zhang$^b$, \quad Shuang Zhao$^a$\\
\small $^a$ Department of Mathematics, Taiyuan University of Technology, Taiyuan 030024, China\\
\small $^b$ College of Mathematics and System Sciences, Xinjiang University, Urumqi 830046, China\\
}
\date{}

\maketitle {\flushleft\bf Abstract:} {\small An edge of a quasi $k$-connected graph is said to be quasi $k$-contractible if the contraction of the edge results in a quasi $k$-connected graph. If every quasi $k$-connected graph without a quasi $k$-contractible edge has either $H_{1}$ or $H_{2}$ as a subgraph, then an unordered pair of graphs $\{H_{1}, H_{2}\}$ is said to be a forbidden pair for quasi $k$-contractible edges. We prove that $\{K_{4}^{-}, \overline{P_{5}}\}$ is a forbidden pair for quasi 5-contractible edges, where $K_{4}^{-}$ is the graph obtained from $K_{4}$ by removing just one edge and $\overline{P_{5}}$ is the complement of a path on five vertices.}

\section{Introduction}
In this paper, all graphs considered are finite, simple and undirected graphs, with undefined terms and notations following \cite{Bondy}. For a graph $G$, let $V(G)$ and $E(G)$ denote the set of vertices and the set of edges of $G$, respectively. For $S\subseteq V(G)$, let $G-S$ denote the graph obtained from $G$ by deleting the vertices of $S$ together with the edges incident with them. The \emph{complement} of a graph $G$ is a graph $\overline{G}$ with the same vertex set as $G$, in which any two distinct vertices are adjacent if and only if they are non-adjacent in $G$. Let $K_{n}$, $P_{n}$ and $C_{n}$ denote the complete graph on $n$ vertices, the path on $n$ vertices, and the cycle on $n$ vertices, respectively. Let $K_{n}^{-}$ be the graph obtained from $K_{n}$ by removing just one edge. For two graphs $G$ and $H$, let $G\cup H$ denote the union of $G$ and $H$ and let $G+H$ denote the join of $G$ and $H$. Moreover, for a positive integer $m$, let $mG$ stand for the union of $m$ copies of $G$.

An edge $e=xy$ of $G$ is said to be \emph{contracted} if it is deleted and its ends are identified. The resulting graph is denoted by $G/e$, and the new vertex in $G/e$ is denoted by $\overline{xy}$. Note that, in the contraction, each resulting pair of double edges is replaced by a single edge. A subgraph of $G$ is said to be \emph{contracted} by identifying each component to a single vertex, removing each of the resulting loops and, finally, replacing each of the resulting double edges by a single edge. Let $k\geq 2$ be an integer and let $G$ be a non-complete $k$-connected graph. An edge or a subgraph of $G$ is said to be \emph{k-contractible} if its contraction results in a $k$-connected graph. A $k$-connected graph without a $k$-contractible edge is said to be a \emph{contraction critical k-connected graph}.

A \emph{cut} of a connected graph $G$ is a subset $V^{\prime}(G)$ of $V(G)$ such that $G-V^{\prime}(G)$ is disconnected. A \emph{k-cut} is a cut of $k$ elements. Suppose $T$ is a $k$-cut of $G$. We say that $T$ is a \emph{nontrivial k-cut}, if the components of $G-T$ can be partitioned into subgraphs $G_{1}$ and $G_{2}$ such that $|V(G_{1})|\geq 2$ and $|V(G_{2})|\geq 2$. A ($k-1$)-connected graph is \emph{quasi k-connected} if it has no nontrivial ($k-1$)-cuts. Clearly, every $k$-connected graph is quasi $k$-connected. Let $G$ be a quasi $k$-connected graph. An edge $e$ of $G$ is said to be \emph{quasi k-contractible} if $G/e$ is still quasi $k$-connected. If $G$ does not have a quasi $k$-contractible edge, then $G$ is said to be a \emph{contraction critical quasi k-connected graph}.

Tutte's \cite{Tutte1961} famous wheel theorem implies that every 3-connected graph on more than four vertices contains an edge whose contraction yields a new 3-connected graph. Thomassen \cite{Thomassen} stated that for $k\geq4$, there are infinitely many contraction critical $k$-connected $k$-regular graphs. Moreover, he studied the forbidden subgraph condition for a $k$-connected graph to have a contractible edge and proved the following theorem.

\begin{thm}\label{thm1}
Every $k$-connected triangle-free graph has a $k$-contractible edge.
\end{thm}

Kawarabayashi~\cite{Kawarabayashi1} proved the following theorem. Theorem~\ref{thm2} is an extension of Theorem~\ref{thm1} in the case of $k$ is odd.

\begin{thm}\label{thm2}
Let $k\geq3$ be an odd integer, and let $G$ be a $k$-connected graph which does not contain $K_{4}^{-}$. Then $G$ has a $k$-contractible edge.
\end{thm}

The same conclusion does not hold when $k$ is even. However, Kawarabayashi~\cite{Kawarabayashi2} proved the following theorem.

\begin{thm}\label{thm3}
Let $k\geq3$ be an integer, and let $G$ be a $k$-connected graph which not contain $K_{4}^{-}$. Then there exists a $k$-contractible edge which is not contained in a triangle or there exists a $k$-contractible triangle.
\end{thm}

The following result due to Ando and Kawarabayashi~\cite{Ando}. They showed that if $s(t-1)<k$, then $\{K_{2}+sK_{1}, K_{1}+tK_{2}\}$ is a forbidden pair for $k$-contractible edges for any $k\geq5$.

\begin{thm}\label{thm4}
Let $k$, $s$ and $t$ be positive integers such that $k\geq5$ and $s(t-1)<k$. If a $k$-connected graph $G$ has neither $K_{2}+sK_{1}$ nor $K_{1}+tK_{2}$, then $G$ has a $k$-contractible edge.
\end{thm}

We focus on quasi 5-connected graphs and obtain the following result.

\begin{figure*}
  \centering
  \subfigure[$K_{4}^{-}$\label{fig1a}]{\includegraphics{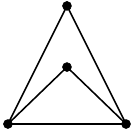}}\hspace{0.8cm}
  \subfigure[$\overline{P_{5}}$\label{fig1b}]{\includegraphics{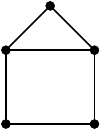}}\hspace{0.8cm}
  \subfigure[$K_{3,3}$\label{fig1c}]{\includegraphics{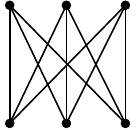}}\hspace{0.8cm}
  \subfigure[cube\label{fig1d}]{\includegraphics{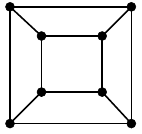}}\hspace{0.8cm}
  \subfigure[$C_{4}^{+}$\label{fig1e}]{\includegraphics{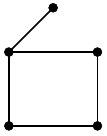}}
  \caption{Graphs $K_{4}^{-}$, $\overline{P_{5}}$, $K_{3,3}$, cube and $C_{4}^{+}$}
  \label{fig1}
\end{figure*}

\begin{thm}\label{thm5}
Let $G$ be a quasi 5-connected graph. If $G$ contains neither $K_{4}^{-}$ nor $\overline{P_{5}}$, then $G$ has a quasi 5-contractile edge.
\end{thm}

\section{Preliminaries}
In this section, we introduce some more definitions and preliminary lemmas.

For a graph $G$, let $E(x)$ denote the set of edges incident with $x\in V(G)$. Let $N_{G}(x)$ denote the set of neighbors of $x\in V(G)$. The degree of $x\in V(G)$ is denoted by $d_{G}(x)$. Let $\delta(G)$ denote the minimum degree of $G$. Let $V_{k}(G)$ denote the set of vertices of degree $k$ in $G$. For $S\subseteq V(G)$, let $N_{G}(S)=\cup_{x\in S}N_{G}(x)-S$ and let $G[S]$ denote the subgraph induced by $S$.

For each integer $n\geq5$, let $C_{n}^{2}$ be the graph obtained from a cycle $C_{n}$ by joining all pairs of vertices of distance two on the cycle. A graph in which each vertex has degree three is called a \emph{cubic graph}. A cubic graph $G$ is called \emph{cyclically 4-connected} if $G$ has four disjoint paths between any two disjoint cycles of $G$. Let $G$ be a graph with nonadjacent edges $e_{1}$ and $e_{2}$. Let $H$ be the graph obtained by subdividing both $e_{1}$ and $e_{2}$ and then adding a new edge connecting the internal vertices of the two paths. This operation is called \emph{adding a handle} to $G$.
The graph $L(G)$ is called the \emph{line graph} of $G$ and is defined as follows. Let $V(L(G))=E(G)$, and for every pair $\{e, f\}\subseteq V(L(G))$, there exists an edge from $e$ to $f$ if and only if they are adjacent edges in $G$.

Let $G$ be a quasi $k$-connected graph and let $E_{0}=\{e\in E(G):G/e$ is $(k-1)$-connected, but not quasi $k$-connected$\}$. For $xy\in E_{0}$, $G/xy$ has a nontrivial $(k-1)$-cut $T^{\prime}$ by the definition of quasi $k$-connected. Furthermore, $\overline{xy}\in T^{\prime}$, for otherwise, $T^{\prime}$ is also a nontrivial $(k-1)$-cut of $G$, contradicts the fact that $G$ is quasi $k$-connected. This implies that $T=(T^{\prime}-\overline{xy})\cup\{x, y\}$ is a $k$-cut of $G$. Moreover, $G-T$ can be partitioned into two subgraphs, each containing more than one vertex. The vertex set of each such subgraph is called a \emph{quasi T-fragment} of $G$ or, briefly, a \emph{quasi fragment}. For an edge $e$ of $G$, a quasi fragment $F$ of $G$ is said to be a \emph{quasi fragment with respect to e} if $V(e)\subseteq N_{G}(F)$. For a set of edges $E^{\prime}\subseteq E(G)$, we say that $F$ is a \emph{quasi fragment with respect to $E^{\prime}$} if $F$ is a quasi fragment with respect to some $e\in E^{\prime}$. A quasi fragment with respect to $e$ or $E^{\prime}$ with least cardinality is called a \emph{quasi atom} with respect to $e$ and $E^{\prime}$ respectively.

The following three lemmas characterize the contraction critical 4-connected graphs. The graphs $K_{3,3}$ and the cube are shown in Figures \ref{fig1c} and \ref{fig1d}, respectively.

\begin{lem}\cite{Martinov1}\label{lem1}
A 4-connected graph is contraction critical if and only if it is 4-regular and each of its edges belongs to a triangle.
\end{lem}

\begin{lem}\cite{Martinov2}\label{lem2}
The only contraction critical 4-connected graphs are $C_{n}^{2}$ for $n\geq5$ and the line graphs of the cubic cyclically 4-connected graphs.
\end{lem}

\begin{lem}\cite{Wormald}\label{lem3}
The class of all cubic cyclically-4-connected graphs can be generated by repeatedly adding handles starting from $K_{3,3}$ and the cube.
\end{lem}

The following lemma is used repeatedly in later proofs.

\begin{lem}\label{lem4}
Let $G$ be a quasi 5-connected graph. If $xy\in E(G)$ and $\delta(G/xy)\geq4$, then $G/xy$ is 4-connected.
\end{lem}

\begin{proof}
Suppose that $G/xy$ is not 4-connected. Then there exists a 3-cut $T^{\prime}$ of $G/xy$. Since $\delta(G/xy)\geq4$, we see that each component of $G/xy-T^{\prime}$ has at least two vertices. Furthermore, $\overline{xy}\in T^{\prime}$, for otherwise, $T^{\prime}$ is also a 3-cut of $G$, a contradiction. Hence, $T=(T^{\prime}-\{\overline{xy}\})\cup\{x, y\}$ is a 4-cut of $G$. And each component of $G-T$ has at least two vertices. It follows that $T$ is a nontrivial 4-cut of $G$, which contradicts the quasi 5-connectivity of $G$.
\end{proof}

\begin{figure}
  \centering
  \includegraphics{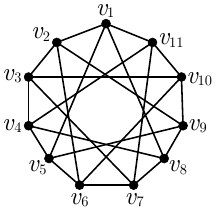}\\
  \caption{The graph $C_{11}^{4}$}\label{fig2}
\end{figure}

Let $C_{11}^{4}$ be a graph as shown in Figure \ref{fig2}. Then $C_{11}^{4}$ is a quasi 5-connected graph. Furthermore, for $i=\{1,2,\dots,11\}$, the edge $v_{i}v_{i+4}$ is quasi 5-contractible, where indices are taken modulo 11. The proof of this fact is straightforward and thus omitted.

\section{The proof of Theorem \ref{lem5}}
In this section, we give a proof of Theorem \ref{lem5}. We first prove several lemmas.

\begin{lem}\label{lem5}
Let $G$ be a contraction critical quasi 5-connected graph that contains neither $K_{4}^{-}$ nor $\overline{P_{5}}$. Let $x\in V_{4}(G)$ with $N_{G}(x)=\{x_{1}, x_{2}, x_{3}, x_{4}\}$ such that $x_{3}x_{4}\in E(G)$. Moreover, for $i=1, 2$, $G/xx_{i}$ is 4-connected. Let $F_{i}$ be a quasi fragment with respect to $xx_{i}$ that contains $x_{j}$, where $\{i, j\}=\{1, 2\}$ and $i\neq j$. Then $G[F_{1}]$ and $G[F_{2}]$ consist of two isolated vertices. Furthermore, for ${a}=F_{2}-\{x_{1}\}$ and ${b}=F_{1}-\{x_{2}\}$, $ab\in E(G)$ holds.
\end{lem}

\begin{proof}
For $i=1,2$, let $T_{i}=N_{G}(F_{i})$ and $\overline{F_{i}}=V(G)-(F_{i}\cup T_{i})$.
Clearly, $x\in T_{1}\cap T_{2}$, $x_{1}\in T_{1}\cap F_{2}$, $x_{2}\in F_{1}\cap T_{2}$ and $\{x_{3}, x_{4}\}\subseteq V(G)-F_{1}-F_{2}$.
Let $X_{1}=(T_{1}\cap F_{2})\cup(T_{1}\cap T_{2})\cup(F_{1}\cap T_{2})$, $X_{2}=(T_{1}\cap F_{2})\cup(T_{1}\cap T_{2})\cup(\overline{F_{1}}\cap T_{2})$, $X_{3}=(\overline{F_{1}}\cap T_{2})\cup(T_{1}\cap T_{2})\cup( T_{1}\cap\overline{F_{2}})$ and $X_{4}=(F_{1}\cap T_{2})\cup(T_{1}\cap T_{2})\cup(T_{1}\cap\overline{F_{2}})$.
The edges $x_{1}x_{3}$, $x_{1}x_{4}$, $x_{2}x_{3}$ and $x_{2}x_{4}$ do not exist, because otherwise $G$ would contain a $K_{4}^{-}$, which contradicts the assumption.
If $x_{1}x_{2}\in E(G)$, then for $i=1, 2$, $d_{G}(x_{i})\geq5$, because $G/xx_{i}$ is 4-connected.

\noindent{\bf Claim 1.} $\overline{F_{1}}\cap T_{2}\neq\emptyset$ and $T_{1}\cap\overline{F_{2}}\neq\emptyset$.

\begin{proof}
We only show that $\overline{F_{1}}\cap T_{2}\neq\emptyset$, and the other one can be handled
similarly. Suppose $\overline{F_{1}}\cap T_{2}=\emptyset$. If $\overline{F_{1}}\cap F_{2}=\emptyset$, then $N_{G}(x_{1})\cap\overline{F_{1}}=\emptyset$, which implies that $T_{1}-\{x_{1}\}$ is a nontrivial 4-cut of $G$, a contradiction. So $\overline{F_{1}}\cap F_{2}\neq\emptyset$. It follows that $|X_{2}|\geq5$ since $N_{G}(x)\cap(\overline{F_{1}}\cap F_{2})=\emptyset$. Then $T_{1}\cap\overline{F_{2}}=\emptyset$ and $|X_{2}|=5$, and thus $|\overline{F_{1}}\cap F_{2}|=1$. Since $|\overline{F_{1}}|\geq2$, $\overline{F_{1}}\cap\overline{F_{2}}\neq\emptyset$. Thus, $|X_{3}|\geq4$, and consequently, $|T_{1}\cap F_{2}|=1$ and $|T_{1}\cap T_{2}|=4$. This implies that $|\overline{F_{1}}\cap\overline{F_{2}}|=1$, and the vertex is adjacent to $x$. Let $\overline{F_{1}}\cap F_{2}=\{a\}$. Then we find that $G[\{a, x, x_{1}, x_{3}, x_{4}\}]\cong\overline{P_{5}}$, a contradiction.
\end{proof}

\noindent{\bf Claim 2.} $|T_{1}\cap F_{2}|=|F_{1}\cap T_{2}|$ and $|T_{1}\cap\overline{F_{2}}|=|\overline{F_{1}}\cap T_{2}|$.

\begin{proof}
We only need to show that $|T_{1}\cap F_{2}|=|F_{1}\cap T_{2}|$.
By contradiction, we assume $|T_{1}\cap F_{2}|>|F_{1}\cap T_{2}|$ without loss of generality. Then $|X_{4}|\leq4$. Furthermore, $|F_{1}\cap T_{2}|\leq2$ by Claim 1. Since $N_{G}(x)\cap(F_{1}\cap\overline{F_{2}})=\emptyset$, $F_{1}\cap\overline{F_{2}}=\emptyset$. If $F_{1}\cap T_{2}=\{x_{2}\}$, then $F_{1}\cap F_{2}\neq\emptyset$. Similar to the proof of $|\overline{F_{1}}\cap F_{2}|=1$ in Claim 1, we obtain $|F_{1}\cap F_{2}|=1$. Let $F_{1}\cap F_{2}=\{a\}$. If $x_{1}x_{2}\in E(G)$, then $G[\{a, x, x_{1}, x_{2}\}]\cong K_{4}^{-}$, a contradiction. If $x_{1}x_{2}\notin E(G)$, then $G[\{a, x, x_{1}, x_{2}\}]\cong C_{4}$ and $N_{G}(a)\cap N_{G}(x_{2})\neq\emptyset$, which implies that $G$ has a $\overline{P_{5}}$, a contradiction. So $|F_{1}\cap T_{2}|=2$. Thus, $|T_{1}\cap F_{2}|=3$ and $|T_{1}\cap T_{2}|=|T_{1}\cap\overline{F_{2}}|=1$. Since $|\overline{F_{2}}|\geq2$ and $|X_{3}|=4$, we have $|\overline{F_{1}}\cap\overline{F_{2}}|=1$. Then, similarly, we can find that $G$ has a $\overline{P_{5}}$, a contradiction.
\end{proof}

\noindent{\bf Claim 3.} $|T_{1}\cap F_{2}|=|F_{1}\cap T_{2}|\geq2$.

\begin{proof}
Suppose that $T_{1}\cap F_{2}=\{x_{1}\}$ and $F_{1}\cap T_{2}=\{x_{2}\}$.
If $F_{1}\cap F_{2}=\emptyset$, then $x_{1}x_{2}\in E(G)$ and $\overline{F_{1}}\cap F_{2}\neq\emptyset$. Since $|X_{2}|=5$ and $N_{G}(x)\cap(\overline{F_{1}}\cap F_{2})=\emptyset$, $|\overline{F_{1}}\cap F_{2}|=1$. Note that $d_{G}(x_{1})\geq5$. It follows that $G$ has a $K_{4}^{-}$, a contradiction.
Therefore, $F_{1}\cap F_{2}\neq\emptyset$. Since $N_{G}(x)\cap(F_{1}\cap F_{2})=\emptyset$, $|X_{1}|\geq5$. By Claim 1, we have $|T_{1}\cap T_{2}|=3$ and $|X_{1}|=5$, which implies $|F_{1}\cap F_{2}|=1$. Let $F_{1}\cap F_{2}=\{a\}$ and $T_{1}\cap T_{2}=\{x, a_{1}, a_{2}\}$. If $x_{1}x_{2}\in E(G)$, then $G[\{a, x, x_{1}, x_{2}\}]\cong K_{4}^{-}$, a contradiction. So $x_{1}x_{2}\notin E(G)$. Furthermore, we see that $N_{G}(a)\cap N_{G}(x_{2})=\emptyset$, for otherwise, $G$ has a $\overline{P_{5}}$ since $G[\{a, x, x_{1}, x_{2}\}]\cong C_{4}$. This implies $F_{1}\cap\overline{F_{2}}\neq\emptyset$. Since $|X_{4}|=5$ and $N_{G}(x)\cap(F_{1}\cap\overline{F_{2}})=\emptyset$, $|F_{1}\cap\overline{F_{2}}|=1$. Let $F_{1}\cap\overline{F_{2}}=\{b\}$. Then we see that $G[\{a, x_{2}, b, a_{1}\}]\cong C_{4}$ and $N_{G}(b)\cap N_{G}(x_{2})\neq\emptyset$. It follows that $G$ has a $\overline{P_{5}}$, a contradiction.
\end{proof}

\noindent{\bf Claim 4.} $|T_{1}\cap\overline{F_{2}}|=|\overline{F_{1}}\cap T_{2}|\geq2$.

\begin{proof}
Suppose $|T_{1}\cap\overline{F_{2}}|=|\overline{F_{1}}\cap T_{2}|=1$. If $\overline{F_{1}}\cap\overline{F_{2}}=\emptyset$, then $(T_{1}\cap\overline{F_{2}})\cup(\overline{F_{1}}\cap T_{2})=\{x_{3}, x_{4}\}$ and $|F_{1}\cap\overline{F_{2}}|=1$. Let $F_{1}\cap\overline{F_{2}}=\{b\}$. Then we see that $G[\{b, x, x_{2}, x_{3}, x_{4}\}]\cong\overline{P_{5}}$, a contradiction. So $\overline{F_{1}}\cap\overline{F_{2}}\neq\emptyset$. Then $|T_{1}\cap T_{2}|\geq2$. By Claim 3, $|T_{1}\cap T_{2}|=|T_{1}\cap F_{2}|=|F_{1}\cap T_{2}|=2$. It follows that $|\overline{F_{1}}\cap\overline{F_{2}}|=1$ and $|\overline{F_{1}}\cap F_{2}|\leq1$.
Regardless of whether $\overline{F_{1}}\cap F_{2}=\emptyset$ or $|\overline{F_{1}}\cap F_{2}|=1$, we can always find a $\overline{P_{5}}$. This proves Claim 4.
\end{proof}

By Claims 3 and 4, $|T_{1}\cap F_{2}|=|F_{1}\cap T_{2}|=|T_{1}\cap\overline{F_{2}}|=|\overline{F_{1}}\cap T_{2}|=2$. Let $T_{1}\cap F_{2}=\{a, x_{1}\}$, $F_{1}\cap T_{2}=\{b, x_{2}\}$, $T_{1}\cap\overline{F_{2}}=\{b_{1}, b_{2}\}$ and $\overline{F_{1}}\cap T_{2}=\{a_{1}, a_{2}\}$. Note that
$|F_{1}\cap F_{2}|\leq1$, $|\overline{F_{1}}\cap F_{2}|\leq1$ and $|F_{1}\cap\overline{F_{2}}|\leq1$.

\noindent{\bf Claim 5.} $F_{1}\cap F_{2}=\emptyset$.

\begin{proof}
Suppose $|F_{1}\cap F_{2}|=1$. Let $F_{1}\cap F_{2}=\{u\}$. Clearly, $x_{1}x_{2}\notin E(G)$, $bx_{1}\notin E(G)$, $bx_{2}\notin E(G)$ and $ax_{2}\notin E(G)$. If $F_{1}\cap\overline{F_{2}}=\emptyset$, then $N_{G}(b)=\{u, a, b_{1}, b_{2}\}$ and $N_{G}(x_{2})=\{u, x, b_{1}, b_{2}\}$, which implies that $G[\{u, b, b_{1}, x_{2}, a\}]\cong\overline{P_{5}}$, a contradiction. Therefore, $|F_{1}\cap\overline{F_{2}}|=1$. Let $F_{1}\cap\overline{F_{2}}=\{v\}$. Then $N_{G}(v)=\{b, x_{2}, b_{1}, b_{2}\}$. Without loss of generality, we assume $bb_{1}\in E(G)$. Then we see that $G[\{u, b, b_{1}, x_{2}, v\}]$ contains a $\overline{P_{5}}$, a contradiction.
\end{proof}

\noindent{\bf Claim 6.} $ax_{1}\notin E(G)$ and $bx_{2}\notin E(G)$.

\begin{proof}
We only show that $bx_{2}\notin E(G)$. Suppose $bx_{2}\in E(G)$. If $|F_{1}\cap\overline{F_{2}}|=1$, let $F_{1}\cap\overline{F_{2}}=\{v\}$. Then $N_{G}(v)=\{b, x_{2}, b_{1}, b_{2}\}$. Since $d_{G}(b)\geq4$, Claim 5 assures us that $bx_{1}\in E(G)$ or $bb_{i}\in E(G)$ for $i=1,2$, which implies that $G$ has either $K_{4}^{-}$ or $\overline{P_{5}}$, a contradiction. So $F_{1}\cap\overline{F_{2}}=\emptyset$. If $x_{1}x_{2}\notin E(G)$, then $\{ax_{2}, bx_{1}\}\subset E(G)$ by Claim 5. Thus, $ab\notin E(G)$. It follows that $N_{G}(b)=\{x_{1}, x_{2}, b_{1}, b_{2}\}$ and $x_{2}b_{i}\in E(G)$ for $i=1,2$. Then we see that $G$ has a $\overline{P_{5}}$, a contradiction. If $x_{1}x_{2}\in E(G)$, then $bx_{1}\notin E(G)$. It follows $N_{G}(b)=\{a, x_{2}, b_{1}, b_{2}\}$. Since $N_{G}(x_{2})\cap\overline{F_{2}}\neq\emptyset$, we may assume $x_{2}b_{1}\in E(G)$ without loss of generality. Since $d_{G}(x_{2})\geq5$, $x_{2}a\in E(G)$ or $x_{2}b_{2}\in E(G)$. No matter the case, $G$ always contains a $K_{4}^{-}$, which is a contradiction.
This proves Claim 6.
\end{proof}

\noindent{\bf Claim 7.} $\overline{F_{1}}\cap F_{2}=\emptyset$ and $F_{1}\cap\overline{F_{2}}=\emptyset$.

\begin{proof}
We only show that $F_{1}\cap\overline{F_{2}}=\emptyset$. Suppose $|F_{1}\cap\overline{F_{2}}|=1$. Let $F_{1}\cap\overline{F_{2}}=\{v\}$. If $\{bb_{1}, bb_{2}\}\subset E(G)$, then $G$ has a $K_{4}^{-}$, a contradiction. Hence, we may assume $N_{G}(b)=\{x_{1}, v, a, b_{1}\}$  without loss of generality. Then we see that $abvx_{2}$ is a cycle of length 4 if $x_{1}x_{2}\notin E(G)$, $x_{1}bvx_{2}$ is a cycle of length 4 if $x_{1}x_{2}\in E(G)$. Furthermore, $b_{1}\in N_{G}(b)\cap N_{G}(v)$, which implies that $G$ has a $\overline{P_{5}}$, a contradiction.
\end{proof}

This completes the proof of Lemma \ref{lem5}.
\end{proof}

By Lemma \ref{lem5}, we have the following lemma.

\begin{lem}\label{lem6}
Let $G$ be a quasi 5-connected graph that contains neither $K_{4}^{-}$ nor $\overline{P_{5}}$. Let $x\in V_{4}(G)$ such that $N_{G}(x)\cong 2K_{1}\cup K_{2}$ or $N_{G}(x)\cong 2K_{2}$. Furthermore, if $N_{G}(x)\cong 2K_{2}$, $N_{G}(x)$ contains two adjacent vertices of degree greater than 4. Then $G$ has a quasi 5-contractible edge.
\end{lem}

\begin{proof}
Assume, to the contrary, that $G$ has no quasi 5-contractible edges. Let $N_{G}(x)=\{x_{1}, x_{2}, x_{3}, x_{4}\}$. If $N_{G}(x)\cong 2K_{1}\cup K_{2}$, then let $x_{3}x_{4}\in E(G)$. If $N_{G}(x)\cong 2K_{2}$, then let $\{x_{1}x_{2}, x_{3}x_{4}\}\subset E(G)$, $d_{G}(x_{1})\geq5$, and $d_{G}(x_{2})\geq5$. For $i=1, 2$, $G/xx_{i}$ is 4-connected by Lemma \ref{lem4}. Then let $F_{i}$ be a quasi fragment with respect to $xx_{i}$, $T_{i}=N_{G}(F_{i})$ and $\overline{F_{i}}=V(G)-(F_{i}\cup T_{i})$. Without loss of generality, assume that $x_{2}\in F_{1}$ and $x_{1}\in F_{2}$. By lemma \ref{lem5}, $|F_{1}|=|F_{2}|=2$. Let $F_{1}=\{x_{2}, b\}$ and let $F_{2}=\{x_{1}, a\}$. Lemma \ref{lem5} assures us that $bx_{2}\notin E(G)$, $ax_{1}\notin E(G)$, and $ab\in E(G)$. If $N_{G}(x)\cong 2K_{2}$, then $G[\{x, x_{1}, x_{2}, a, b\}]\cong\overline{P_{5}}$, a contradiction. So $N_{G}(x)\cong 2K_{1}\cup K_{2}$.

Note that $N_{G}(a)\cong 4K_{1}$. Otherwise, $G$ contains either $K_{4}^{-}$ or $\overline{P_{5}}$, a contradiction. This implies that $G/ax_{2}$ is 4-connected by Lemma \ref{lem4}. Let $F_{3}$ be a quasi fragment with respect to $ax_{2}$. Let $T_{3}=N_{G}(F_{3})$ and $\overline{F_{3}}=V(G)-(F_{3}\cup T_{3})$. Clearly, $x_{1}\in T_{3}$. It follows that $x\notin T_{3}$. Otherwise, either $N_{G}(x)\cap F_{3}=\emptyset$ or $N_{G}(x)\cap\overline{F_{3}}=\emptyset$. In either case, this would imply that $T_{3}-\{x\}$ is a nontrivial 4-cut of $G$, a contradiction. Without loss of generality, we assume $x\in F_{3}$. If $|N_{G}(a)\cap\overline{F_{3}}|=1$, then $(T_{3}-\{a, x_{1}\})\cup(N_{G}(a)\cap\overline{F_{3}})$ is a 4-cut of $G$, and thus, $|\overline{F_{3}}|=2$. However, it is clear that $G$ contains a $\overline{P_{5}}$. So $|N_{G}(a)\cap\overline{F_{3}}|=2$ and $|N_{G}(a)\cap F_{3}|=1$. If $b\in F_{3}$, then $N_{G}(x_{2})\cap\overline{F_{3}}=\emptyset$, a contradiction. So $b\in\overline{F_{3}}$.

Let $N_{G}(a)=\{x_{2}, b, a_{1}, a_{2}\}$. Without loss of generality, assume that $a_{1}\in F_{3}$ and $a_{2}\in\overline{F_{3}}$. Then we see that $(N_{G}(a)\cup N_{G}(x_{1})\cup N_{G}(x_{2}))\cap F_{3}=\{a_{1}, x\}$. It follows $|F_{3}|\leq3$. If $F_{3}=\{a_{1}, x\}$, then $T_{3}=\{a, x_{1}, x_{2}, x_{3}, x_{4}\}$. Hence, $\{a_{1}x_{3}, a_{1}x_{4}\}\subseteq E(G)$, implying $G[\{x, x_{3}, x_{4}, a_{1}\}]\cong K_{4}^{-}$, a contradiction. Therefore, $|F_{3}|=3$. Without loss of generality, we may assume $F_{3}=\{a_{1}, x, x_{3}\}$. Consequently, $x_{4}\in T_{3}$. Let $\{t\}=T_{3}-\{a, x_{1}, x_{2}, x_{4}\}$. If $a_{1}x_{4}\in E(G)$, then $G[\{x, x_{3}, x_{4}, a_{1}\}]\cong K_{4}^{-}$, leading to a contradiction. So $a_{1}t\in E(G)$. However, this implies that $G[\{a_{1}, x_{3}, x, x_{2}, t\}]\cong\overline{P_{5}}$, which is also a contradiction.
\end{proof}

\begin{lem}\label{lem7}
Let $G$ be a contraction critical quasi 5-connected graph that contains neither $K_{4}^{-}$ nor $\overline{P_{5}}$. Let $x\in V_{4}(G)$ such that $N_{G}(x)\cong 4K_{1}$. Let $F_{1}$ be a quasi atom with respect to $E(x)$. Then,\\
{\rm(i)} $G[F_{1}]$ consists of two isolated vertices, each of degree four;\\
{\rm(ii)} If $u$ is a common neighbor of two vertices in $F_{1}$, then $|N_{G}(u)\cap N_{G}(x)|\leq2$.
\end{lem}

\begin{proof}
Let $N_{G}(x)=\{x_{1}, x_{2}, x_{3}, x_{4}\}$. Let $T_{1}=N_{G}(F_{1})$ and let $\overline{F_{1}}=V(G)-(F_{1}\cup T_{1})$. Without loss of generality, we assume that $F_{1}$ is a quasi fragment with respect to $xx_{1}$ and $x_{2}\in F_{1}$. Let $F_{2}$ be a quasi fragment with respect to $xx_{2}$ and let $T_{2}=N_{G}(F_{1})$, $\overline{F_{2}}=V(G)-(F_{2}\cup T_{2})$.
Clearly, $x\in T_{1}\cap T_{2}$ and $x_{2}\in F_{1}\cap T_{2}$.
Let $X_{1}=(T_{1}\cap F_{2})\cup(T_{1}\cap T_{2})\cup(F_{1}\cap T_{2})$, $X_{2}=(T_{1}\cap F_{2})\cup(T_{1}\cap T_{2})\cup(\overline{F_{1}}\cap T_{2})$, $X_{3}=(\overline{F_{1}}\cap T_{2})\cup(T_{1}\cap T_{2})\cup( T_{1}\cap\overline{F_{2}})$, $X_{4}=(F_{1}\cap T_{2})\cup(T_{1}\cap T_{2})\cup(T_{1}\cap\overline{F_{2}})$.

\noindent{\bf Claim 1.} {\rm(i)} If $F_{1}\cap F_{2}\neq\emptyset$, then $\overline{F_{1}}\cap\overline{F_{2}}\neq\emptyset$ and $|F_{1}\cap F_{2}|=1$;\\
\indent~~~~~~~~~~{\rm(ii)} If $F_{1}\cap\overline{ F_{2}}\neq\emptyset$, then $\overline{F_{1}}\cap F_{2}\neq\emptyset$ and $|F_{1}\cap\overline{F_{2}}|=1$.

\begin{proof}
We only prove that {\rm(i)} holds. If $F_{1}\cap F_{2}\neq\emptyset$, then $|X_{1}|\geq5$. It follows $|F_{1}\cap T_{2}|\geq|T_{1}\cap\overline{F_{2}}|$. If $\overline{F_{1}}\cap\overline{F_{2}}=\emptyset$, then $|\overline{F_{2}}|<|F_{1}|$, a contradiction. So $\overline{F_{1}}\cap\overline{F_{2}}\neq\emptyset$. This implies $|X_{3}|\geq4$.
If $|F_{1}\cap F_{2}|\geq2$, then $|X_{1}|\geq6$. Otherwise, $F_{1}\cap F_{2}$ is a quasi fragment with respect to $xx_{2}$ and $|F_{1}\cap F_{2}|<|F_{1}|$, which contradicts the choice of $F_{1}$. Hence, $|X_{1}|=6$ and $|X_{3}|=4$. So $|\overline{F_{1}}\cap\overline{F_{2}}|=1$, implying $|\overline{F_{2}}|<|F_{1}|$, a contradiction. So $|F_{1}\cap F_{2}|=1$.
\end{proof}

\noindent{\bf Claim 2.} $|F_{1}\cap T_{2}|\geq2$.

\begin{proof}
Suppose $F_{1}\cap T_{2}=\{x_{2}\}$. If $F_{1}\cap F_{2}=\emptyset$, then $F_{1}\cap\overline{F_{2}}\neq\emptyset$ and $T_{1}\cap F_{2}\neq\emptyset$. By Claim 1, $|F_{1}\cap\overline{F_{2}}|=1$. This implies $|X_{4}|\geq5$, and thus $|T_{1}\cap F_{2}|=1$. Let $T_{1}\cap F_{2}=\{a\}$ and $F_{1}\cap\overline{F_{2}}=\{b\}$. Then $a\neq x_{1}$ since $ax_{2}\in E(G)$. Therefore, $G[\{b, x, x_{1}, x_{2}\}]\cong C_{4}$ and $N_{G}(b)\cap N_{G}(x_{2})\neq\emptyset$. This implies that $G$ has a $\overline{P_{5}}$, a contradiction. So $F_{1}\cap F_{2}\neq\emptyset$.
Similarly, we have $F_{1}\cap\overline{F_{2}}\neq\emptyset$. By Claim 1, $|F_{1}\cap F_{2}|=|F_{1}\cap\overline{F_{2}}|=1$. Let $F_{1}\cap F_{2}=\{a\}$ and $F_{1}\cap\overline{F_{2}}=\{b\}$.
Note that $G[F_{1}]$ is connected by the choice of $F_{1}$.

If $T_{1}\cap F_{2}=\emptyset$, then $|\overline{F_{1}}\cap F_{2}|\leq1$, which implies $|F_{2}|<|F_{1}|$, a contradiction. So $T_{1}\cap F_{2}\neq\emptyset$. If $|T_{1}\cap F_{2}|\geq2$, then $|X_{4}|\leq4$, which is impossible. Therefore, $|T_{1}\cap F_{2}|=1$. Similarly, $|T_{1}\cap\overline{F_{2}}|=1$. Thus, $|T_{1}\cap T_{2}|=3$. Let $T_{1}\cap T_{2}=\{x, c_{1}, c_{2}\}$. Then $ax_{2}bc_{1}$ and $ax_{2}bc_{2}$ are cycles of length four. Moreover, $N_{G}(a)\cap N_{G}(x_{2})\neq\emptyset$ or $N_{G}(b)\cap N_{G}(x_{2})\neq\emptyset$. Consequently, $G$ has either $K_{4}^{-}$ or $\overline{P_{5}}$ as a subgraph, a contradiction.
\end{proof}

\noindent{\bf Claim 3.} $|F_{1}\cap T_{2}|=2$.

\begin{proof}
Suppose $|F_{1}\cap T_{2}|\neq2$. By Claim 2, it follows that $|F_{1}\cap T_{2}|\geq3$. Without loss of generality, assume $x_{1}\in T_{1}\cap F_{2}$. If $|T_{1}\cap F_{2}|\leq2$, then $|X_{2}|\leq4$, which implies $\overline{F_{1}}\cap F_{2}=\emptyset$. Consequently, $|F_{2}|<|F_{1}|$, a contradiction. Therefore, $|T_{1}\cap F_{2}|\geq3$. It follows that $|T_{1}\cap\overline{F_{2}}|\leq1$. Then $|X_{3}|\leq3$, and thus $\overline{F_{1}}\cap\overline{F_{2}}=\emptyset$. This implies $|\overline{F_{2}}|<|F_{1}|$, a contradiction.
\end{proof}

\noindent{\bf Claim 4.} $|T_{1}\cap F_{2}|=|T_{1}\cap\overline{F_{2}}|=2$.

\begin{proof}
we first prove that $|T_{1}\cap F_{2}|\geq2$. Assume, for contradiction, that $|T_{1}\cap F_{2}|\leq1$.
If $\overline{F_{1}}\cap F_{2}=\emptyset$, then $|F_{2}|<|F_{1}|$ by Claim 3, a contradiction. Thus, $\overline{F_{1}}\cap F_{2}\neq\emptyset$, implying that $|T_{1}\cap F_{2}|=1$ and  $|\overline{F_{1}}\cap F_{2}|=1$. Let $T_{1}\cap F_{2}=\{a\}$ and $\overline{F_{1}}\cap F_{2}=\{b\}$. We observe that $F_{1}\cap F_{2}\neq\emptyset$, for otherwise, $G[\{a, b, x, x_{2}\}]\cong C_{4}$ and $N_{G}(a)\cap N_{G}(b)\neq\emptyset$, which implies that $G$ has a $\overline{P_{5}}$, a contradiction. By Claim 1, we have $|F_{1}\cap F_{2}|=1$. It follows that $|T_{1}\cap T_{2}|\geq2$. Note that $\overline{F_{1}}\cap T_{2}\neq\emptyset$. Otherwise, $|\overline{F_{1}}\cap\overline{F_{2}}|=1$, which implies $|\overline{F_{1}}|<|F_{1}|$, a contradiction. Hence, $|T_{1}\cap T_{2}|=2$ and $|\overline{F_{1}}\cap T_{2}|=1$. Let $F_{1}\cap F_{2}=\{c\}$ and let $T_{1}\cap T_{2}=\{x, u\}$. Then $bacu$ is a cycle of length four, and either $N_{G}(b)\cap N_{G}(a)\neq\emptyset$ or $N_{G}(b)\cap N_{G}(c)\neq\emptyset$. This implies that $G$ has either $K_{4}^{-}$ or $\overline{P_{5}}$ as a subgraph, a contradiction.
Therefore, $|T_{1}\cap F_{2}|\geq2$. Similarly, we can show that $|T_{1}\cap\overline{F_{2}}|\geq2$. Hence, $|T_{1}\cap F_{2}|=|T_{1}\cap\overline{F_{2}}|=2$.
\end{proof}

Without loss of generality, we assume $x_{1}\in T_{1}\cap F_{2}$. Let $T_{1}\cap F_{2}=\{a, x_{1}\}$, $F_{1}\cap T_{2}=\{b, x_{2}\}$, $T_{1}\cap\overline{F_{2}}=\{b_{1}, b_{2}\}$ and $\overline{F_{1}}\cap T_{2}=\{a_{1}, a_{2}\}$.

\noindent{\bf Claim 5.} $F_{1}\cap F_{2}=\emptyset$, $bx_{2}\notin E(G)$ and $F_{1}\cap\overline{F_{2}}=\emptyset$.

\begin{proof}
Unless $|F_{1}\cap\overline{F_{2}}|=1$ and the vertex is adjacent to $x$, Claim 5 holds, similar to the proof in Lemma \ref{lem5}. Next, we consider this special case.
Let $F_{1}\cap\overline{F_{2}}=\{v\}$. Then $N_{G}(v)=\{b, x, b_{1}, b_{2}\}$. Suppose $F_{1}\cap F_{2}\neq\emptyset$. Then $|F_{1}\cap F_{2}|=1$ by Claim 1. Let $F_{1}\cap F_{2}=\{u\}$. Then $N_{G}(u)=\{a, b, x_{1}, x_{2}\}$. We observe that $bx_{1}\notin E(G)$, $bx_{2}\notin E(G)$, and $b$ cannot be adjacent to both $b_{1}$ and $b_{2}$ simultaneously. Otherwise, $G$ has either $K_{4}^{-}$ or $\overline{P_{5}}$ as a subgraph, a contradiction. Without loss of generality, assume $N_{G}(b)=\{u, v, a, b_{1}\}$. Since $d_{G}(x_{2})\geq4$, either $x_{2}a\in E(G)$ or $x_{2}b_{1}\in E(G)$. However, in either case, $G$ has a $\overline{P_{5}}$, a contradiction. Therefore, $F_{1}\cap F_{2}=\emptyset$. Suppose $bx_{2}\in E(G)$. Then both $bx_{2}xx_{1}$ and $bx_{2}xv$ are cycles of length 4. Since $d_{G}(b)\geq4$, we have $N_{G}(b)\cap N_{G}(x_{2})\neq\emptyset$, which implies that $G$ has a $\overline{P_{5}}$, a contradiction. Thus, $bx_{2}\notin E(G)$. This proves Claim 5.
\end{proof}

By Claim 5, we have $F_{1}=\{b, x_{2}\}$. Furthermore, $N_{G}(b)=\{x_{1}, a, b_{1}, b_{2}\}$ and $N_{G}(x_{2})=\{a, x, b_{1}, b_{2}\}$. This proves {\rm(i)} holds. Clearly, $b_{1}b_{2}\notin E(G)$, for otherwise, $G[\{b, b_{1}, b_{2}, x_{2}\}]\cong K_{4}^{-}$, a contradiction. If $ax_{1}\in E(G)$, then $N_{G}(b)\cong K_{2}\cup 2K_{1}$, and thus, $G$ has a quasi 5-contractible edge by Lemma \ref{lem6}, a contradiction. So $ax_{1}\notin E(G)$. Since $N_{G}(x)\cap\overline{F_{2}}\neq\emptyset$, we have $N_{G}(x)\cap(\overline{F_{1}}\cap\overline{F_{2}})\neq\emptyset$. Without loss of generality, assume $x_{4}\in\overline{F_{1}}\cap\overline{F_{2}}$. In the following, we show that  $x_{3}\in\overline{F_{1}}\cap F_{2}$, which implies that {\rm(ii)} holds.

Suppose $x_{3}\notin\overline{F_{1}}\cap F_{2}$. Then $|\overline{F_{1}}\cap F_{2}|\leq1$. If $|\overline{F_{1}}\cap F_{2}|=1$, then we see that $G$ has either $K_{4}^{-}$ or $\overline{P_{5}}$, a contradiction. So $\overline{F_{1}}\cap F_{2}=\emptyset$. It follows that $N_{G}(a)=\{b, x_{2}, a_{1}, a_{2}\}$ and $N_{G}(x_{1})=\{b, x, a_{1}, a_{2}\}$.
By Lemma \ref{lem4}, $G/ax_{2}$ is 4-connected. Let $F_{3}$ be a quasi fragment with respect to $ax_{2}$. Let $T_{3}=N_{G}(F_{3})$ and $\overline{F_{3}}=V(G)-(F_{3}\cup T_{3})$. Clearly, $x_{1}\in T_{3}$. If $x\in T_{3}$, then without loss of generality, we assume that $\{x_{3}, b_{1}\}\subseteq F_{3}$ and $\{x_{4}, b_{2}\}\subseteq\overline{F_{3}}$. Consequently, $T_{3}=\{a, x_{2}, x_{1}, x, b\}$. Then we may assume that $a_{1}\in F_{3}$ and $a_{2}\in\overline{F_{3}}$ without loss of generality. Since $d_{G}(x_{3})\geq4$, $|F_{3}|\geq4$, which implies that $\{a_{1}, b_{1}, x_{3}\}$ is a 3-cut of $G$, a contradiction. Therefore, $x\notin T_{3}$.

Without loss of generality, we assume $x\in F_{3}$. Similar to the second paragraph of Lemma \ref{lem6}, we have that $|N_{G}(a)\cap F_{3}|=1$, $|N_{G}(a)\cap\overline{F_{3}}|=2$, and $b\in\overline{F_{3}}$. Without loss of generality, assume that $a_{1}\in F_{3}$ and $a_{2}\in\overline{F_{3}}$. It follows $|F_{3}|\leq3$. Since $N_{G}(x)\cong 4K_{1}$, $F_{3}=\{a_{1}, x\}$, and thus $T_{3}=\{a, x_{1}, x_{2}, x_{3}, x_{4}\}$. Moreover, $N_{G}(a_{1})=\{a, x_{1}, x_{3},x_{4}\}$. By a similar argument for $bx_{1}$, we have $N_{G}(b_{i})=\{b, x_{2}, x_{3},x_{4}\}$ for some $i\in\{1, 2\}$. Without loss of generality, assume $N_{G}(b_{1})=\{b, x_{2}, x_{3},x_{4}\}$. If $|V(G)|\geq11$, then $|V(G)|=12$, for otherwise, $\{a_{2}, b_{2}, x_{3}, x_{4}\}$ forms a nontrivial 4-cut of $G$, a contradiction. However, we observe that $G$ has a $\overline{P_{5}}$. Thus, $|V(G)|=11$. If $a_{2}b_{2}\notin E(G)$, then $N_{G}(a_{2})=\{a, x_{1}, x_{3}, x_{4}\}$, and thus $N_{G}(a_{1})=N_{G}(a_{2})$. This implies that $N_{G}(a_{1})$ is a nontrivial 4-cut of $G$, a contradiction. Therefore, $a_{2}b_{2}\in E(G)$. It follows $G\cong C_{11}^{4}$. Thus, $G$ has quasi 5-contractible edges, a contradiction. This proves that $x_{3}\in\overline{F_{1}}\cap F_{2}$, thus completing the proof of the lemma \ref{lem7}.
\end{proof}

By Lemma \ref{lem7}, we have the following lemma.

\begin{lem}\label{lem8}
Let $G$ be a quasi 5-connected graph that contains neither $K_{4}^{-}$ nor $\overline{P_{5}}$. If $G$ has a vertex $x\in V_{4}(G)$ such that $N_{G}(x)\cong 4K_{1}$, then $G$ has a quasi 5-contractible edge.
\end{lem}

\begin{proof}
Suppose that $G$ has no quasi 5-contractible edges. Let $N_{G}(x)=\{x_{1}, x_{2}, x_{3}, x_{4}\}$. Let $F$ be a quasi atom with respect to $E(x)$. Without loss of generality, assume that $F$ is a quasi fragment with respect to $xx_{1}$ and $x_{2}\in F$. By Lemma \ref{lem7}, $G[F]$ consists of two isolated vertices, each of degree four. Let $F=\{x_{2}, b\}$.
Let $T=N_{G}(F)=\{x, x_{1}, a_{1}, a_{2}, a_{3}\}$ and let $\overline{F}=V(G)-(F\cup T)$. Then $N_{G}(x_{2})=\{x, a_{1}, a_{2}, a_{3}\}$ and $N_{G}(b)=\{x_{1}, a_{1}, a_{2}, a_{3}\}$.
Thus, $N_{G}(b)\cong 4K_{1}$. Otherwise, $G$ has either $K_{4}^{-}$ or $\overline{P_{5}}$ as a subgraph, a contradiction.

Let $C$ be a quasi atom with respect to $E(b)$ and let $R=N_{G}(C)$, $\overline{C}=V(G)-(C\cup R)$. By Lemma \ref{lem7}, $G[C]$ consists of two isolated vertices, each of degree four.
If $C$ is a quasi fragment with respect to $bx_{1}$, then $x_{2}\in R$. Thus, $C=\{a_{i}, x\}$, which implies that $N_{G}(a_{i})=\{b, x_{2}, x_{3}, x_{4}\}$. However, by Lemma \ref{lem7} {\rm(ii)}, this is impossible. Therefore, we assume that $C$ is a quasi fragment with respect to $ba_{1}$ without loss of generality. Since $N_{G}(a_{1})\cap N_{G}(x_{2})=\emptyset$, $x_{2}\notin R$. This implies that $x_{1}\in C$, and consequently, $x\in R$. Without loss of generality, we assume that $C=\{x_{1}, x_{3}\}$. Let $R=\{b, a_{1}, x, t_{1}, t_{2}\}$. Then $N_{G}(x_{1})=\{b, x, t_{1}, t_{2}\}$ and $N_{G}(x_{3})=\{a_{1}, x, t_{1}, t_{2}\}$.

Let $D$ be a quasi fragment with respect to $bx_{1}$. Let $Q=N_{G}(D)$ and $\overline{D}=V(G)-(D\cup Q)$. Clearly, $x_{2}\in Q$. If $x\in Q$, we assume that $x_{3}\in D$ without loss of generality. It then follows that $\{t_{1}, t_{2}\}\subset D\cup Q$, which implies that $Q-\{x_{1}\}$ is a nontrivial 4-cut of $G$, a contradiction. Hence, $x\notin Q$. We assume that $x\in D$ and $t_{1}\in\overline{D}$ without loss of generality. Then $x_{3}\in Q$. Similar to the second paragraph of Lemma \ref{lem6}, we have that $|N_{G}(b)\cap D|=1$ and $|N_{G}(b)\cap\overline{D}|=2$.
If $a_{1}\in D$, then $t_{2}\in D$. Otherwise, $D=\{x, a_{1}\}$, which implies that $N(a_{1})=\{b, x_{2}, x_{3}, x_{4}\}$, a contradiction. It follows $|D|\leq4$. If $|D|=4$, then $D=\{a_{1}, t_{2}, x, x_{4}\}$. Hence, $N_{G}(a_{1})\cap N_{G}(x)=\{x_{2}, x_{3}, x_{4}\}$, a contradiction. So $D=\{a_{1}, x, t_{2}\}$. However, we see that $G[\{a_{1}, t_{2}, x_{1}, b, x_{3}\}]\cong\overline{P_{5}}$, a contradiction. If $a_{2}\in D$ or $a_{3}\in D$, we also can obtain a contradiction that $G$ has either $K_{4}$ or $\overline{P_{5}}$ similarly. This complete the proof.
\end{proof}

Let $C_{4}^{+}$ denote the graph shown in Figure \ref{fig1e}. Now, we are prepared to prove Theorem \ref{thm5}.

\begin{proof}[{\bf Proof of Theorem \ref{thm5}}]
By way of contradiction, we assume that there is a contraction critical quasi 5-connected graph $G$ that contains neither $K_{4}^{-}$ nor $\overline{P_{5}}$. By Theorem \ref{thm2}, we have $\kappa(G)=4$. Thus, $\delta(G)=4$. Then we see that for any vertex of degree four in $G$, its neighbor set is isomorphic to $4K_{1}$ or $2K_{1}\cup K_{2}$ or $2K_{2}$, for otherwise, $G$ contains either $K_{4}^{-}$ or $\overline{P_{5}}$, a contradiction. By Lemmas \ref{lem6} and \ref{lem8}, the neighbor set of each vertex of degree four is isomorphic to $2K_{2}$. Moreover, for any two adjacent vertices in the neighbor set, at least one of them has degree four.

Suppose there exists a vertex $x\in V_{4}(G)$ with $N_{G}(x)=\{x_{1}, x_{2}, x_{3}, x_{4}\}$ such that $\{x_{1}x_{2}, x_{3}x_{4}\}\subset E(G)$ and $d_{G}(x_{1})\geq5$. Then $d_{G}(x_{2})=4$.
By Lemma \ref{lem4}, $G/xx_{2}$ is 4-connected. Since $G$ has no quasi 5-contractible edges, $G/xx_{2}$ has a nontrivial 4-cut. That is, $G$ has a 5-cut $T$ such that $\{x, x_{2}\}\subset T$. And $G-T$ can be partitioned into two subgraphs, say $G[F]$ and $G[\overline{F}]$, where each subgraph has at least two vertices. Without loss of generality, we assume $x_{1}\in F$ and $\{x_{3}, x_{4}\}\subseteq T\cup\overline{F}$. Since $d_{G}(x_{2})=4$, we see that $(N_{G}(x)\cup N_{G}(x_{2}))\cap F=\{x_{1}\}$. Thus, $|F|=2$. Let $F=\{x_{1}, a\}$. Then $d_{G}(a)=4$. However, we observe that $N_{G}(a)\not\cong 2K_{2}$, a contradiction.

Hence, all neighbors of a vertex with degree four also have degree four. This implies that $G$ is 4-connected, 4-regular, and every edge of $G$ is in a triangle. By Lemmas \ref{lem1} and \ref{lem2}, $G$ is isomorphic to either $C_{n}^{2}$ for $n\geq5$ or the line graph of a cubic cyclically 4-connected graph. For $n\geq8$, it is straightforward to verify that $C_{n}^{2}$ has nontrivial 4-cuts. On the other hand, for  $n=5, 6, 7$, $C_{n}^{2}$ contains a $\overline{P_{5}}$. Therefore, $G\cong L(G^{\ast})$, where $G^{\ast}$ is a cubic cyclically 4-connected graph. By Lemma \ref{lem3}, $G^{\ast}$ is obtained by repeatedly adding handles starting from $K_{3,3}$ or the cube. Both $K_{3,3}$ and the cube contain $C_{4}^{+}$ as a subgraph. Moreover, if a graph contains $C_{4}^{+}$ as a subgraph, then the graph obtained from it by adding a handle also contains $C_{4}^{+}$ as a subgraph. Thus, $G^{\ast}$ has $C_{4}^{+}$ as a subgraph. Note that $L(C_{4}^{+})\cong\overline{P_{5}}$. This implies that $G$ has a $\overline{P_{5}}$, a contradiction.
This completes the proof of Theorem \ref{thm5}.
\end{proof}

\end{document}